\documentclass[a4paper]{article}

\usepackage{verbatim}
\usepackage{hyperref}
\usepackage{amsmath}
\usepackage{enumerate, amsthm}
\usepackage{mathrsfs}
\usepackage{xcolor}
\usepackage[margin=1 in]{geometry}
\usepackage{mathtools}
\usepackage{tikz}
\usetikzlibrary{patterns}
\usetikzlibrary{arrows}
\usepackage{ifthen}
\usepackage{bm}
\usepackage{amssymb}
\usepackage{centernot}

\usepackage{caption}
\usepackage{subcaption}

\numberwithin{equation}{section}

\theoremstyle{plain}
\newtheorem{theorem}{Theorem}[section]
\newtheorem{lemma}[theorem]{Lemma}

\newtheorem{proposition}[theorem]{Proposition}

\theoremstyle{definition}

\theoremstyle{remark}
\newtheorem{remark}[theorem]{Remark}

 \newcommand{\R}{\mathbb R}
 \newcommand{\Z}{\mathbb Z}
 \newcommand{\N}{\mathbb N}
 \renewcommand{\H}{\mathbb H}

 \newcommand{\cR}{\mathcal R}
 \newcommand{\cS}{\mathcal S}
 \newcommand{\cD}{\mathcal D}
 \newcommand{\cB}{\mathcal B}

\newcommand{\Lc}{\Lambda}

\newcommand{\lr}[4]{#3\xleftrightarrow[#1]{#2} #4}
\newcommand{\nlr}[4]{#3\centernot{\xleftrightarrow[#1]{#2}} #4}

\title{Slab percolation for the Ising model revisited}
\author{Franco Severo\footnotemark[1]\footnote{ETH Z\"{u}rich, franco.severo@math.ethz.ch}}

\date{}

\begin{document}
\thispagestyle{empty}
\maketitle

\begin{abstract}
In this note, we give a new and short proof for a theorem of Bodineau stating that the slab percolation threshold $\hat{p}_c$ for the FK-Ising model coincides with the standard percolation critical point $p_c$ in all dimensions $d\geq3$. Both proofs rely on the positivity of the surface tension for $p>p_c$ proved by Lebowitz \& Pfister.
The key difference is that while Bodineau's proof is based on a delicate dynamic renormalization inspired by the work of Barsky, Grimmett \& Newman, our proof utilizes a technique of Benjamini \& Tassion to prove the uniqueness of macroscopic clusters via sprinkling, which then implies percolation on slabs through a rather straightforward static renormalization.
\end{abstract}

\section{Introduction}\label{sec:intro}

We study the supercritical phase of the FK (also known as random-cluster) model with cluster weight $q\geq1$ on $\Z^d$, $d\geq 3$. This class of percolation models was introduced by Fortuin \& Kasteleyn \cite{ForKas72} and is intimately linked to the $q$-state Potts model for integers $q\geq2$ -- see \cite{Gri06,DumNotes17} for an account on these models and their connections. For $q=2$ this model is sometimes called the FK-Ising model.

The model is defined as follows. Let $E$ be the set of nearest neighbour edges of $\Z^d$. For a finite $\Lambda\subset \Z^d$, denote by $E(\Lambda)\subset E$ the set of edges of $\Z^d$ intersecting $\Lambda$.
Let $q\geq1$, $p\in[0,1]$ and $\xi\in \{0,1\}^{E\setminus E(\Lambda)}$, and consider the probability measure $\phi^\xi_{\Lambda,p,q}$ on $\Omega_\Lambda\coloneqq \{0,1\}^{E(\Lambda)}$ given by
\begin{equation*}
    \phi^\xi_{\Lambda,p,q}(\omega)\coloneqq \frac{1}{Z^\xi_{\Lambda,p,q}} p^{o(\omega)}(1-p)^{c(\omega)}q^{k^\xi_\Lambda(\omega)},
\end{equation*}
where $o(\omega)$ and $c(\omega)$ denote the number of open ($=1$) and closed ($=0$) edges in $\omega$, $k^\xi_\Lambda(\omega)$ denotes the number of clusters intersecting $\Lambda$ in the concatenation $\omega\cup \xi \in \{0,1\}^E$, and $Z^\xi_{\Lambda,p,q}$ is a renormalizing constant. We often refer to $\xi\equiv 1$ and $\xi\equiv 0$ as \emph{wired} and \emph{free} boundary conditions. By monotonicity in boundary conditions, we can construct the infinite volume measure $\phi^0_{p,q}$ (resp. $\phi^1_{p,q}$) on $\Omega\coloneqq \{0,1\}^E$ by taking a weak limit of $\phi^0_{\Lambda,p,q}$ (resp. $\phi^1_{\Lambda,p,q}$) as $\Lambda\uparrow\Z^d$.
We define the \emph{percolation threshold}
\begin{equation*}
    p_c=p_c(q,d)\coloneqq \inf\{p\in[0,1]:~\phi^1_{p,q}(\lr{}{}{0}{\infty})>0\}.
\end{equation*}
For $d\geq3$, consider the ``slab box'' of size $N$ and thickness $L$ given by
\begin{equation*}
    \cS(L,N)\coloneqq \{-L,\ldots,L\}^{d-2}\times \{-N,\ldots,N\}^2.
\end{equation*}
We define the \emph{slab percolation threshold} 
\begin{equation*}
    \hat{p}_c=\hat{p}_c(q,d)\coloneqq \inf \big\{ p\in[0,1]:~\exists L\geq0 \text{ such that } \inf_{N}\inf_{x\in \cS(L,N)} \phi^0_{\cS(L,N),p,q}(\lr{}{}{0}{x})>0 \big\}.
\end{equation*}
It is clear that $\hat{p}_c\geq p_c$. The slab percolation threshold was introduced in the seminal work of Pisztora \cite{Pis96}, where a powerful coarse graining technique was developed to describe the behavior of the model assumption that $p>\hat{p}_c$. This technique has then found multiple applications in the study of fundamental features of supercritical percolation, such as the Wulff crystal construction \cite{Bod99,CerPis00,CerPis01}, the structure of translation invariant Gibbs measures for the Ising model \cite{Bod06}, the exponential decay of truncated Ising correlations \cite{DGR20} and the existence of long range order for the random the field Ising model \cite{DLX22}, to cite just a few.

It is conjectured that $\hat{p}_c(q,d)=p_c(q,d)$ for all $q\geq1$ and $d\geq 3$, which implies that Pisztora's coarse graining and its consequences are valid up to the critical point. This has only been proved for $q=1$ (corresponding to Bernoulli percolation) by Grimmett \& Marstrand \cite{GriMar90} and for $q=2$ (corresponding to the Ising model) by Bodineau \cite{Bod05} -- see however \cite{DumTas19} for a weaker result for integer $q$. 
In this note, we give a new proof for the Ising case.

\begin{theorem}\label{thm:main_Ising}
    One has $\hat{p}_c(2,d)=p_c(2,d)$ for all $d\geq 3$.
\end{theorem}

Similarly to \cite{Bod05}, a crucial ingredient in our proof is the positivity of the surface tension. 
Consider the rectangle
\begin{equation*}
    \cR(L,M)\coloneqq \{-L,\ldots,L\}^{d-1} \times \{-M,\ldots,M\}
\end{equation*}
and its top and bottom sides given by $\partial^{top} \cR(L,M)= \{-L,\ldots,L\}^{d-1} \times \{M\}$ and $\partial^{bot} \cR(L,M)= \{-L,\ldots,L\}^{d-1} \times \{-M\}$.
We define the (free) \emph{surface tension} (in the direction $e_d$) as
\begin{equation}\label{eq:def_tau}
    \tilde{\tau}_p=\tilde{\tau}_p(q,d)\coloneqq \sup_{C\geq1}\sup_{\delta>0}\, \limsup_{L\to\infty} \frac{1}{L^{d-1}} \log \Big( \phi^0_{\Lambda_{CL},p,q} \big[ \nlr{}{}{\partial^{bot} \cR(L,\delta L)}{\partial^{top} \cR(L,\delta L)} \big]^{-1} \Big),
\end{equation}
where $\Lambda_N\coloneqq \{-N,\ldots,N\}^d$.
We can then consider the associated critical point
\begin{equation*}
    \tilde{p}_c=\tilde{p}_c(q,d)\coloneqq \inf\{p\in[0,1]:~ \tilde{\tau}_p(q,d)>0\}.
\end{equation*}
It is easy to see that $\tilde{p}_c\geq p_c$. 
Lebowitz \& Pfister \cite{LebPfi81} proved that in the Ising case $q=2$ one has $\tau_p>0$ for all $p>p_c$, where $\tau_p$ is defined similarly but with \emph{wired} boundary conditions on $\cR(L,\delta L)$ instead -- see Appendix~\ref{sec:appendix}. Using a weak mixing property, one can compare wired and free boundary conditions, thus leading to the following.

\begin{theorem}\label{thm:tilde_p_c=p_c}
    One has $\tilde{p}_c(2,d)=p_c(2,d)$ for all $d\geq3$.
\end{theorem}
 
For the sake of completeness, we include in Appendix~\ref{sec:appendix} the proof of positivity of the (wired) surface tension from \cite{LebPfi81} along with the comparison between boundary conditions leading to Theorem~\ref{thm:tilde_p_c=p_c}. The latter follows the same lines as \cite[Theorem 3.1]{Bod05}, with a slight simplification due to the fact that in our definition of $\tilde{p}_c$ the (free) boundary conditions are allowed to be at a macroscopic distance from the support of the relevant (dis)connection event, which is not the case in \cite{Bod05}.

Theorem~\ref{thm:main_Ising} then follows readily from Theorem~\ref{thm:tilde_p_c=p_c} and the following result, which concerns all FK models with $q\geq1$.

\begin{theorem}\label{thm:main_FK}
   One has $\hat{p}_c(q,d)=\tilde{p}_c(q,d)$ for all $q\geq1$ and $d\geq 3$. 
\end{theorem}

Theorem~\ref{thm:main_FK} is similar to \cite[Theorem 2.2]{Bod05}, but our proof is completely different. We also stress that, due to the aforementioned difference between our definition of $\tilde{p}_c$ and that of \cite{Bod05}, our result is slightly stronger. Indeed, the proof of \cite[Thorem 2.2]{Bod05} relies on a delicate \emph{dynamic renormalization} scheme inspired by the work of Barsky, Grimmett \& Newman \cite{BarGriNew91} on Bernoulli percolation in the half space, which requires connections to go up to the boundary of the domain (with free boundary conditions). Our approach on the other hand is based on \emph{static renormalization}, for which a macroscopic distance from the boundary is typically harmless.

Our proof of Theorem~\ref{thm:main_FK} goes as follows. First, we observe that the surface order exponential cost for disconnection given by the condition $p>\tilde{p}_c$ implies that the clusters in $\Lambda_L$ touching $\partial \Lambda_L$ are $\ell$-dense in $\Lambda_L$ with high probability for $\ell=C (\log L)^{\frac{1}{d-1}} = o(\log L)$. Then we adapt a technique of Benjamini \& Tassion \cite{benjaminitassion17} to prove that all the clusters crossing an annulus get connected to each other with high probability after an $\varepsilon$-Bernoulli sprinkling -- see also \cite{DS23} for a very similar use of this technique for Voronoi percolation and \cite{DGRS19,DGRST23b} for more sophisticated arguments in the context of Gaussian Free Field level sets and random interlacements. Finally, we use this local uniqueness property to perform a standard static renormalization argument implying percolation in the slab.

\begin{remark}
For simplicity, we chose to focus on nearest-neighbor interactions, but it is straightforward to adapt all the proofs to finite range interactions as well.
\end{remark}

\paragraph{Acknowledgements.} I would like to thank Hugo Duminil-Copin, Ulrik Thinggaard Hansen and the anonymous referees for their helpful comments on earlier drafts of this paper. This research was supported by the European Research Council (ERC) under the European Union’s Horizon 2020 research and innovation program (grant agreement No 851565).

\section{From disconnection to slab percolation}

In this section we prove Theorem~\ref{thm:main_FK}. We fix $q\geq1$ and $d\geq3$ and henceforth omit them from the notation. As explained in Section~\ref{sec:intro}, the proof is split into two parts, which are done in separate subsections.

\subsection{Uniqueness with sprinkling}

We start with the following lemma, which is an easy consequence of the FKG inequality.

\begin{lemma}\label{lem:surface_free_box}
   If $\tilde{\tau}_p>0$, then there exists $\delta>0$ and such that for infinitely many $L$ one has 
    \begin{equation}\label{eq:surface_free_box}
        \phi^0_{\Lambda_{L},p}[\lr{}{}{\Lambda_\ell}{\partial \Lambda_{\delta L}}]\geq 1-e^{-\delta \ell^{d-1}} ~~~ \forall\, 1\leq \ell\leq \delta L.
    \end{equation}
\end{lemma}

\begin{proof}
   Since $\tilde{\tau}_p>0$, there exist $C'\geq1$ and $\delta'>0$ such that for infinitely many $L$ one has
   \begin{equation}\label{eq:surface_free}
       \phi^0_{\Lambda_{C'L},p}[ \nlr{}{}{\partial^{bot} \cR(L, \delta' L)}{\partial^{top} \cR(L, \delta' L)} ] \leq e^{-\delta' L^{d-1} }.
   \end{equation} 
   Cover the hyperplane $\{-C'L,\ldots,C'L\}^{d-1}\times \{0\}$ by $m\leq C''({L}/{\ell})^{d-1}$ boxes $\Lambda_\ell(x_1),\ldots,\Lambda_\ell(x_m)$. Notice that $$\bigcap_{i=1}^{m} \{\nlr{}{}{\Lambda_\ell(x_i)}{\Lambda_{\delta' L}(x_i)}\} \subset \{\nlr{}{}{\partial^{bot} \cR(L, \delta' L)}{\partial^{top} \cR(L, \delta' L)}\}.$$ Therefore, by the FKG inequality together with \eqref{eq:surface_free}, we can find an $i_0\in\{1,\ldots,m\}$ such that 
   $$\phi^0_{\Lambda_{C'L},p}[\nlr{}{}{\Lambda_\ell(x_{i_0})}{\partial \Lambda_{\delta' L}(x_{i_0})}]\leq e^{-\delta' \ell^{d-1}/C''}.$$ 
   By comparison between boundary conditions we obtain $$\phi^0_{\Lambda_{2C'L}(x_i),p}[\nlr{}{}{\Lambda_\ell(x_{i_0})}{\partial \Lambda_{\delta' L}(x_{i_0})}]\leq e^{-\delta' \ell^{d-1}/C''},$$ 
   which does not depend on $x_{i_0}$, so the result follows with $\delta=\delta' \min\{1/2C', 1/C''\}$.
\end{proof}

Given $p\in[0,1]$, $\varepsilon>0$, $\Lambda\subset \Z^d$ and boundary conditions $\xi$ on $\Lambda$, let $\omega \sim_d \phi^\xi_{\Lambda,p}$ and $\gamma \sim_d \otimes_{x\in E(\Lambda)} \textrm{Ber}(\varepsilon)$ be independent random variables, and denote its joint distribution on $\Omega_\Lambda^2$ by $\psi^\xi_{\Lambda,p,\varepsilon}$.

The following proposition is the heart of our proof. It is inspired by the work of Benjamini \& Tassion \cite{benjaminitassion17}, where it was proved that any ``everywhere percolating'' subgraph of $\Z^d$ becomes connected after an $\varepsilon$-Bernoulli sprinkling. We adapt their proof to the case of a very dense subgraph. 

\begin{proposition}\label{prop:unique}
    Let $\mathrm{Unique}(L)$ be the event that there is a cluster in $\omega\cap \Lambda_{L}$ crossing $\Lambda_L\setminus \Lambda_{L/8}$ and that every cluster of $\omega\cap \Lambda_L$ crossing $\Lambda_{L/2}\setminus \Lambda_{L/4}$ are connected to each other in $(\omega\cup\gamma)\cap \Lambda_{L/2}$.
    For $p>\tilde{p}_c$ there exists $\delta>0$ such that for all $\varepsilon>0$ one has
    \begin{equation*}
       \limsup_{L\to\infty} \inf_\xi \psi^\xi_{\Lambda_{L},p,\varepsilon}[\mathrm{Unique}(\delta L)]\to 1.
    \end{equation*}
\end{proposition}

\begin{proof}
    We follow closely the proof and notation of \cite[Proposition 4.1]{DS23}, where a similar result is proved for Voronoi percolation.
    Fix $p>\tilde{p}_c$ and $\varepsilon>0$. 
    Let $\delta'>0$ be given by Lemma~\ref{lem:surface_free_box} and $L'$ such that \eqref{eq:surface_free_box} holds. Set $L\coloneqq2L'$, $\delta\coloneqq\delta'/4$, $C_0\coloneqq (d\delta')^{-\frac{1}{d-1}}$ and 
    $$\ell=\ell(L,p)\coloneqq C_0(\log L)^{\frac{1}{d-1}}.$$ 
    By monotonicity on boundary conditions and the inequality \eqref{eq:surface_free_box}, we obtain $$\psi_{\Lambda_L}^\xi[\lr{}{\omega}{\Lc_\ell(x)}{\partial \Lambda_{\delta L}}] \geq \phi_{\Lambda_{L'}}^0[\lr{}{}{\Lc_\ell}{\partial \Lambda_{\delta' L'}}]\geq 1-L^{-d}$$ 
    for all $x\in\Lambda_{\delta L}$ and all $\xi$. 
    By union bound, one concludes that the event 
    $$\mathcal{A}_L\coloneqq \bigcap_{x\in \ell\Z^d\cap\Lc_{\delta L}} \{\lr{}{\omega}{\Lc_\ell(x)}{\partial \Lc_{\delta L}}\}$$
    satisfies $$\limsup_{L\to\infty} \inf_\xi \psi_{\Lambda_L}^\xi[\mathcal{A}_L]=1.$$ 
    It remains to prove that $\mathrm{Unique}(\delta L)$ happens with high probability under $\psi^\xi$ (uniformly in $\xi$) conditionally on $\mathcal{A}_L$.
    First, notice that the existence of a cluster of $\omega\cap \Lambda_{\delta L}$ crossing $\Lambda_{\delta L}\setminus \Lambda_{\delta L/8}$ is automatically implied by $\mathcal{A}_L$, so we only need to focus on the uniqueness part. 
    Consider the set of boundary clusters
    \begin{equation*} 
    \mathcal{C}\coloneqq \{C\subset \Lc_{\delta L}:~ \text{$C$ is a cluster in $\omega\cap\Lc_{\delta L}$ such that $C\cap\partial \Lc_{\delta L/2}\neq\emptyset$}\}.
    \end{equation*} 
    Obviously, $|\mathcal{C}|\leq |\partial \Lambda_{\delta L/2}|\leq CL^{d-1}$. Given $\omega\subset \eta \subset E(\Z^d)$, we define the relation $\sim_\eta$ in $\mathcal{C}$ by setting
    \begin{equation*}
    C\sim_\eta C' \text{ if } \lr{\Lc_{\delta L/2}}{\eta}{C}{C'}.
    \end{equation*}
    It is enough to prove that, with high probability conditionally on $\mathcal{A}_L$, all the clusters in $\mathcal{C'}\coloneqq\{C\in\mathcal{C}:~ C\cap\Lc_{\delta L/4}\neq\emptyset\}$ are connected to each other in $(\omega\cup \gamma) \cap\Lc_{\delta L}$, or equivalently $|\mathcal{C}'/\sim_{\omega\cup\gamma}|=1$.
    
    Let $V_i\coloneqq \Lc_{\delta L/2-i\sqrt{L}}$, $0\leq i\leq \lfloor\delta\sqrt{L}/4\rfloor$. We will add $\gamma$ to $\omega$ progressively on each annulus $V_{i}\setminus V_{i+1}$. For every $0\leq i\leq \lfloor\delta \sqrt{L}/4\rfloor$, set 
    \begin{equation*}
        \eta_i\coloneqq \omega\cup (\gamma\cap (V_0\setminus V_i)).
    \end{equation*}
    Given $\eta\supset \omega$, let
    $$\mathcal{U}_i(\eta)\coloneqq \mathcal{C}_i/\sim_\eta \,,$$
    where 
    $$\mathcal{C}_i\coloneqq \{C\in\mathcal{C}:~ C\cap V_i\neq\emptyset\}.$$
    Finally, we set 
    $$U_i\coloneqq |\mathcal{U}_i(\eta_i)|.$$
    Notice that it is enough to prove that $U_{\lfloor\delta \sqrt{L}/4\rfloor}=1$ with high probability conditionally on $\mathcal{A}_L$, which follows from the following lemma.

    \begin{lemma}\label{lem:gluing}
        There exists $c>0$ such that for every $\xi$ and every $0\leq i\leq \lfloor \delta\sqrt{L}/4\rfloor-8$, one has
        \begin{equation}\label{eq:gluing}
            \psi_{\Lambda_L,p,\varepsilon}^\xi[\{U_{i+8}> 1\vee U_i/2\} \cap \mathcal{A}_L]\leq e^{-cL^{1/4}}.
        \end{equation}
    \end{lemma}
    \noindent  By Lemma~\ref{lem:gluing} together with a union bound, the following event occurs with high probability
    \[\bigcap_{0\leq i\leq \lfloor\sqrt{L}/4\rfloor-8} \{U_{i+8}\le  1\vee U_i/2\} \cap \mathcal{A}_L.\]
    On this event, $U_{\lfloor\delta\sqrt{L}/4\rfloor}>1$ would imply $U_0\ge 2^ {\lfloor\delta\sqrt{L}/32\rfloor-1}$, which contradicts the fact that $U_0\leq CL^{d-1}$. This yields that $U_{\lfloor\delta\sqrt{L}/4\rfloor}=1$ on this event, thus concluding the proof. 
\end{proof}

\begin{figure}[!h]
    \centering
    \includegraphics[width=0.65\textwidth]{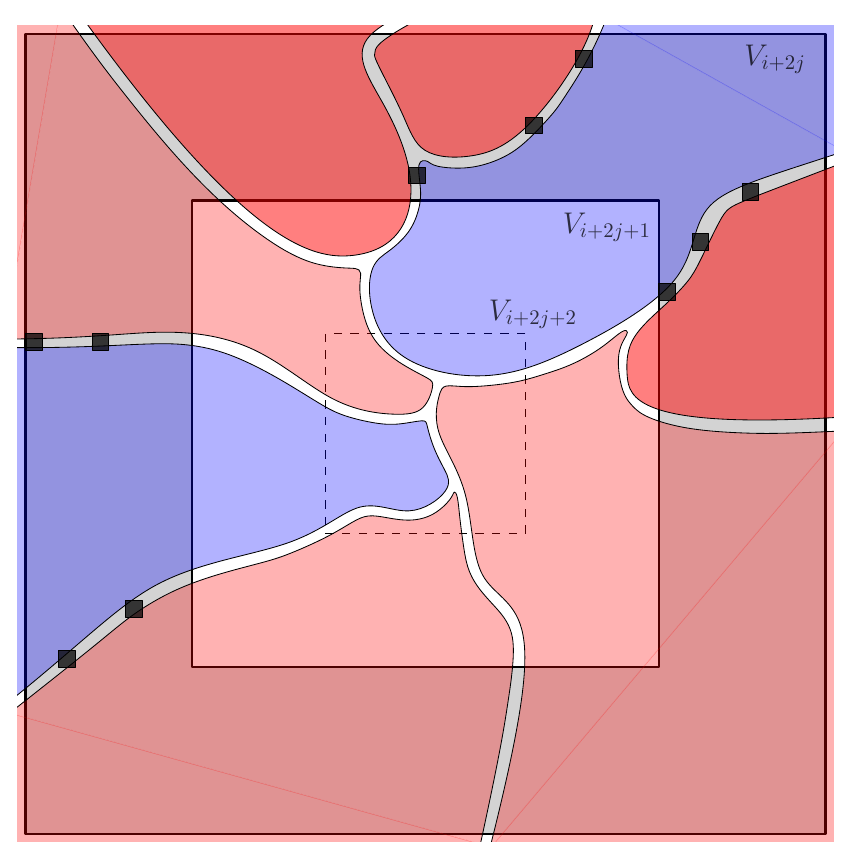}
    \caption{\label{steering} The gluing procedure of Lemma~\ref{lem:gluing}. The blue and red clusters represent $\tilde{\mathcal{U}}_1$ and $\tilde{\mathcal{U}}_2$, respectively. The dark red clusters represent the set $\tilde{C}$: we are therefore in the case where $\tilde{C}\cap V_{i+2j+1}\neq\emptyset$, and $A$ is the grey annulus. The black squares represent the boxes  $(D_k)_{k=1}^n$, of radius $2\ell =o(\log L)$, where $\tilde{\mathcal{U}}_1$ and $\tilde{\mathcal{U}}_2$ meet.}
\end{figure}

It remains to give the following proof.

\begin{proof}[Proof of Lemma~\ref{lem:gluing}]
    We fix $\xi$ and write $\psi$ instead of $\psi^\xi_{\Lambda_L,p,\varepsilon}$ for simplicity; all estimates will be uniform on $\xi$. 
    Fix $0\leq i\leq \lfloor \delta\sqrt{L}/4\rfloor-8$. 
    We start by finding a sub-annulus of $V_i\setminus V_{i+8}$ such that most clusters in it are crossing. For every $\eta\supset \omega$ and $j\in\{0,1,2,3\}$, let
    $$\mathcal{U}_i^j(\eta)\coloneqq \{C\in\mathcal{U}_i(\eta):~ C\cap V_{i+2j}\neq\emptyset \text{ and } C\cap V_{i+2j+2}=\emptyset\}.$$
    In the definition above, we abuse the notation by identifying the equivalence class of cluster $C\in\mathcal{U}_i(\eta)$ with its associated $\eta$-cluster. Since $(\mathcal{U}_i^j(\eta))_j$ are disjoint subsets of $\mathcal{U}_i(\eta)$, we can find $j\in\{0,1,2,3\}$ such that $|\mathcal{U}_i^j(\eta_i)|\leq |\mathcal{U}_i(\eta_i)|/4=U_i/4$. We fix such a $j$ for the rest of the proof and focus on the annulus $V_{i+2j}\setminus V_{i+2j+2}$. 
    
    We now further restrict to one of the two sub-annuli $V_{i+2j}\setminus V_{i+2j+1}$ and $V_{i+2j+1}\setminus V_{i+2j+2}$ as follows. If at least one of the clusters in $\mathcal{U}_i^j(\eta_i)$ touches $V_{i+2j+1}$, we ``wire'' all clusters in $\mathcal{U}_i^j(\eta_i)$ (i.e.~we treat their union as a single element) and focus on the annulus $V_{i+2j}\setminus V_{i+2j+1}$. Otherwise, we forget about all the clusters in $\mathcal{U}_i^j(\eta_i)$ and focus on the annulus $V_{i+2j+1}\setminus V_{i+2j+2}$.
    More precisely, consider the family
    \begin{equation*}
        \tilde{\mathcal{U}}\coloneqq \left\{ \begin{array}{ll}
    \{C\in\mathcal{U}_i(\eta_i):~ C\cap V_{i+2j+2}\neq\emptyset\}\cup \big\{\tilde{C}\big\},     &  ~~\text{if } \tilde{C}\cap V_{i+2j+1}\neq \emptyset, \\
    \{C\in\mathcal{U}_i(\eta_i):~ C\cap V_{i+2j+2}\neq\emptyset\},     & ~~\text{otherwise.}
    \end{array}
    \right.
    \end{equation*}
    where $\tilde{C}\coloneqq \cup_{C\in\mathcal{U}_i^j(\eta_i)} C$ if $\mathcal{U}_i^j(\eta_i)\neq \emptyset$ and $\tilde{C}=\emptyset$ otherwise, and the annulus
    \begin{equation*}
        A\coloneqq \left\{ \begin{array}{ll} V_{i+2j}\setminus V_{i+2j+1},  & ~~\text{if } \tilde{C}\cap V_{i+2j+1}\neq \emptyset, \\
    V_{i+2j+1}\setminus V_{i+2j+2}    & ~~\text{otherwise.}\end{array}\right.
    \end{equation*}
Notice that in any case, every element of  $\tilde{\mathcal{U}}$ contains a crossing of the annulus $A$, and that every box $\Lc_\ell(x)\subset A$, $x\in\ell\Z^d$, intersects at least one element of $\tilde{\mathcal{U}}$. 
Define the following partition of $\tilde{ \mathcal U}$. Let $\overline {\mathcal U}_0$ be the set made of all the elements of $\tilde{ \mathcal U}$ connected to $\tilde C$ in $\eta_{i+8}$ and $\overline{ \mathcal U}_1,\dots,\overline {\mathcal U}_m$ be the partition of $\tilde {\mathcal U}\setminus \overline{ \mathcal U}_0$ induced by the equivalence relation $\sim_{\eta_{i+8}}$. 
If we assume that $|\overline {\mathcal U}_k|\ge 4$ for all $1\le k \le m $, then we get
    \[U_{i+8}\le\frac 1 4 |\tilde {\mathcal U}\setminus \overline{ \mathcal U}_0|+|\mathcal{U}_i^j(\eta_i)|\le \frac 1 2U_i.\]
This cannot occur on the event $\mathcal{E}\coloneqq \{U_{i+8}> 1\vee U_i/2\} \cap \mathcal{A}_L$.
As a result, if $\mathcal{E}$ happens, then there exists a non-trivial partition $\tilde{\mathcal{U}}=\tilde{\mathcal{U}}_1\cup \tilde{\mathcal{U}}_2$ such that $|\tilde{\mathcal{U}}_1|< 4$ and $\nlr{A}{\eta_{i+8}}{\tilde{\mathcal{U}}_1}{\tilde{\mathcal{U}}_2}$, where here we abuse the notation again by identifying $\tilde{\mathcal{U}}_1$ and $\tilde{\mathcal{U}}_2$ with the union of the associated $\eta_i$-clusters.
As a conclusion, we have
\begin{equation}\label{eq:gluing_proof1}
    \psi[\mathcal{E}]\leq \sum_{\tilde{U}} \psi[\{\tilde{\mathcal{U}}=\tilde{U}\}\cap\mathcal{A}_L] \, \Big( \sum_{\substack{\tilde{U}=\tilde{U}_1\sqcup \tilde{U}_2 \\ 1\leq|\tilde U_1|< 4}} \psi[\,\nlr{A}{\eta_{i+8}}{\tilde U_1}{\tilde U_2} \,|\, \{\tilde{\mathcal{U}}=\tilde{U}\}\cap\mathcal{A}_L] \Big),
\end{equation}
where the sum in $\tilde{U}$ runs over all the (finitely many) possibilities for $\tilde{\mathcal{U}}$ which are compatible with $\mathcal{A}_L$.
    
Fix a set $\tilde{U}$ and a partition $\tilde{U}=\tilde{U}_1\sqcup \tilde{U}_2$ as in \eqref{eq:gluing_proof1}. Recall that by construction both $\tilde{U}_1$ and $\tilde{U}_2$ contain a crossing of the annulus $A$ and their union is $\ell$-dense in $A$. Therefore, by splitting $A$ into annuli of thickness $10\ell$, one can find vertices $x_1,\ldots,x_n\in  \ell \Z^d$, $n\geq c_d\sqrt{L}/\ell$, such that the boxes $D_k\coloneqq \Lambda_{2\ell}(x_k)\subset A$, $j\in\{1,\ldots,n\}$, are disjoint and each one intersects both $\tilde{U}_1$ and $\tilde{U}_2$. In particular, on each of these boxes, there exists a path of length at most $4d\ell$ between $\tilde{U}_1$ and $\tilde{U}_2$ which is fully open in $\gamma$ with probability at least $\varepsilon^{4d\ell}$.
Notice that the event $\{\tilde{\mathcal{U}}=\tilde{U}\}\cap \mathcal{A}_L$ is $(\omega,\eta_i)$-measurable and therefore independent of $\gamma\cap (V_i\setminus V_{i+8})$.
By independence, we conclude that
\begin{equation}\label{eq:gluing_proof2}
    \begin{split}
        \psi[\, \nlr{A}{\eta_{i+8}}{\tilde{U}_1}{\tilde{U}_2} \,|\, \{\tilde{\mathcal{U}}=\tilde{U}\}\cap\mathcal{A}_L] &\leq \psi\Big[\bigcap_{k=1}^{n} \nlr{D_k}{\gamma}{\tilde{U}_1}{\tilde{U}_2}\Big]
        \leq \big(1-\varepsilon^{4d\ell}\big)^n
        \leq e^{-cL^{1/4}},
        \end{split}
\end{equation}
for some constant $c>0$, where we have used that $n \varepsilon^{4d\ell}\gg L^{1/4}$.  
Combining \eqref{eq:gluing_proof1}, \eqref{eq:gluing_proof2} and the fact that $|\tilde U|\leq CL^{d-1}$, we obtain
$\psi[\mathcal{E}]\leq \sum_{k=1}^3 {CL^{d-1}\choose k} e^{-cL^{1/4}}\leq e^{-c'L^{1/4}}$,
for some $c'>0$ and $L$ large enough, thus concluding the proof. 
\end{proof}

\subsection{Static renormalization}

\begin{proposition}\label{prop:slab}
    For every $p>\tilde{p}_c$ and $\varepsilon>0$, there exists $L\geq1$ such that
    \begin{equation*}
        \inf_{N}\inf_{x\in \cS(L,N)} \psi^0_{\cS(L,N),p,\varepsilon} (\lr{}{\omega\cup \gamma}{0}{x})>0.
    \end{equation*}
\end{proposition}

\begin{proof}
    On the slab box $\cS(L,N)$, consider the sites of the form $x(u)=(0,\ldots,0, \frac{\delta}{8}Lu)$ for $u\in B_n\coloneqq \{-n,\ldots,n\}^2$, where $n\coloneqq \lfloor \frac{8(N-L)}{\delta L}\rfloor$ is chosen so that $\Lambda_L(x(u))\subset \cS(L,N)$ for all $u\in B_n$. We construct a site percolation model $\eta$ on $B_n$ under the measure $\psi^0_{\cS(L,N),p,\varepsilon}$ as follows. We set $\eta_u=1$ if the event $\mathrm{Unique}(\delta L)$ happens when centered at $x(u)$. First notice that, by the definition of $\mathrm{Unique}(\delta L)$, a path of $\eta$-open sites from $u$ to $v$ induces a path in $\omega\cup\gamma$ from $\Lambda_{\delta L/8}(x(u))$ to $\Lambda_{\delta L/8}(x(v))$. Second, by the Markov property one has
    \begin{equation*}
        \psi^0_{\cS(L,N),p,\varepsilon}[\eta_u=1 \,\big|\, (\eta_v:\, \|u-v\|_\infty\geq 16/\delta)] \geq \alpha(L) ~~~\text{a.s. } \forall u\in B_n,
    \end{equation*}
    where $\alpha(L)\coloneqq \inf_\xi \psi^\xi_{\Lambda_{L},p,\varepsilon}[\mathrm{Unique}(\delta L)]$. Since $\limsup_{L\to\infty} \alpha(L)= 1$ by Proposition~\ref{prop:unique}, the main result of \cite{LigSchSta97} implies that there exists an $L$ (which we fix now) such that $\eta$ stochastically dominates a product measure with parameter $s>p_c^{site}(\Z^2)$. In particular, there exists a constant $c>0$ such that $\psi^0_{\cS(L,N),p,\varepsilon}[\lr{}{\eta}{u}{v}]\geq c$ for all $u,v\in B_n$ and all $N$. Now, given any point $x\in\cS(L,N)$, there exists $u\in B_n$ such that $x\in \Lambda_{2L}(x(u))$, so that the event $\{\lr{}{\eta}{0}{u}\}\cap \{\Lambda_L \text{ fully open in } \gamma\} \cap \{\Lambda_{2L}(x(u))\cap\cS(L,N) \text{ fully open in } \gamma\}$ is contained in $\{\lr{}{\omega\cup\gamma}{0}{x}\}$. As a conclusion, we obtain the desired uniform lower bound $\psi^0_{\cS(L,N),p,\varepsilon} [\lr{}{\omega\cup \gamma}{0}{x}]\geq c'\coloneqq c \varepsilon^{CL^d}>0$ for every $x\in \cS(L,N)$ and every $N$.
\end{proof}

\begin{lemma}\label{lem:stoc_dom}
    For every $0\leq p<p'\leq 1$, there exists $\varepsilon>0$ such that for every $\Lambda\subset \Z^d$ one has the stochastic domination $\omega_{p'}\succ \omega_p\cup \gamma_\varepsilon$, where $\omega_{p'} \sim_d \phi^0_{\Lambda,p'}$ and $(\omega_{p},\gamma_\varepsilon) \sim_d \psi^0_{\Lambda,p,\varepsilon}$. 
\end{lemma}
\begin{proof}
    We construct a natural coupling between $\omega_p$ and $\omega_{p'}$ as follows. Let $e_1,\ldots,e_n$ be an arbitrary enumeration of $E(\Lambda)$ and $(U_i)_{i\geq n}$ be i.i.d.~$\mathrm{Uniform}[0,1]$ random variables. We inductively define 
    $$\omega_p(e_{i}) = 1_{ \{ U_i\leq \phi^0_{\Lambda,p}[\omega(e_i)\,|\, \omega(e_j)=\omega_p(e_j)~\forall j<i] \} },$$
    and analogous for $p'$. By the Markov property, the desired stochastic domination would directly follow if we prove that there exists $\varepsilon>0$ such that for every $\Delta\subset \Z^d$, every pair of boundary condition $\xi'\geq \xi$ and every edge $e\in E(\Delta)$, we have
    $$\phi^{\xi'}_{\Delta,p'}[\omega_e]-\phi^\xi_{\Delta,p}[\omega_e]\geq \varepsilon.$$
    By monotonicity in boundary conditions, it is enough to consider $\xi'=\xi$. By Russo's formula, one has
    $$\partial_p \phi^\xi_{\Delta,p}[\omega_e] = \frac{1}{p(1-p)} \sum_{f\in E(\Delta)} \phi^\xi_{\Delta,p}[\omega_e \omega_f]-\phi^\xi_{\Delta,p}[\omega_e]\phi^\xi_{\Delta,p}[\omega_f]\geq \frac{1}{p(1-p)} \phi^\xi_{\Delta,p}[\omega_e](1-\phi^\xi_{\Delta,p}[\omega_e])\geq 1/q,$$ where we used in the last inequality that $\frac{p}{p+q(1-p)}=\phi^0_{\{e\},p}[\omega_e]\leq\phi^\xi_{\Delta,p}[\omega_e]\leq \phi^1_{\{e\},p}[\omega_e] = p$. The result follows for $\varepsilon=(p'-p)/q$.
\end{proof}

\begin{proof}[Proof of Theorem~\ref{thm:main_FK}]
    It follows readily from Proposition~\ref{prop:slab} and Lemma~\ref{lem:stoc_dom} that $\hat{p}_c\leq \tilde{p}_c$. The opposite inequality is classical, see \cite{Pis96}.
\end{proof}


\appendix

\section{Positivity of the surface tension}\label{sec:appendix}

In this Appendix, we prove Theorem~\ref{thm:tilde_p_c=p_c}. The proof relies on specific properties of the Ising model, namely the Ginibre inequality, which is used in the proof of Theorem~\ref{thm:pos_suf_tens} below, and the uniqueness of infinite volume measure for its FK representation on the half space with positive boundary wiring, which is used in the proof of Proposition~\ref{prop:surface_tension_wired=free} below (see Lemma~\ref{lem:wired=free}). Since we will only work with the FK-Ising model ($q=2$), we henceforth omit $q$ from the notation. We start by introducing the Ising model. Given a finite $\Lambda\subset\Z^d$, and inverse temperature $\beta\geq 0$ and a boundary field $\eta\in \R^{\partial\Lambda}$, we define the Hamiltonian on $\Sigma_\Lambda\coloneqq \{-1,1\}^{\Lambda}$ given by
\begin{equation*}
    H^\eta_\Lambda(\sigma)\coloneqq - \sum_{xy\in E(\Lambda)} \sigma_x\sigma_y,
\end{equation*}
where for $y\in\partial \Lambda$ we write $\sigma_y=\eta_y$. Given an inverse temperature $\beta\geq0$, consider probability measure $\langle\cdot\rangle^\eta_{\Lambda,\beta}$ on $\Sigma_\Lambda$ given by
\begin{equation*}
    \langle f(\sigma) \rangle^\eta_{\Lambda,\beta} = \frac{1}{Z^\eta_{\Lambda,\beta}} \sum_{\sigma\in\Sigma_\Lambda} f(\sigma)e^{-\beta H^\eta_\Lambda(\sigma)},
\end{equation*}
where $Z^\eta_{\Lambda,\beta}=\sum_{\sigma\in\Sigma_\Lambda} e^{-\beta H^\eta_\Lambda(\sigma)}$ is the normalizing partition function -- see e.g.~\cite{FV17} for an introduction to the Ising model.
We denote by $+$ the boundary field $\eta\equiv 1$ and by $\pm$ the boundary field given by $\eta=1_{\partial^+\Lambda}-1_{\partial^-\Lambda}$, where $\partial^+\Lambda=\partial^\Lambda\cap (\Z^{d-1}\times\N)$ and $\partial^-\Lambda=\partial\Lambda\setminus \partial^+\Lambda$. Let $\langle \cdot\rangle^+_\beta$ denote the infinite volume measure on $\Sigma=\{-1,1\}^{\Z^d}$ obtained as the weak limit of $\langle \cdot\rangle^+_{\Lambda,\beta}$ as $\Lambda\uparrow\Z^d$. The critical inverse temperature is given by 
\begin{equation*}
    \beta_c\coloneqq \inf \{\beta\geq0:~ \langle \sigma_0\rangle^+_\beta>0\}.
\end{equation*}
Finally, the \emph{surface tension} at $\beta$ is defined as 
\begin{equation}\label{eq:def_surface_tension}
    \tau_\beta \coloneqq \lim_{L\to\infty} \frac{1}{L^{d-1}} \lim_{M\to\infty} \log \frac{Z^+_{\cR(L,M),\beta}}{Z^\pm_{\cR(L,M),\beta}}.
\end{equation}
It is classical that the limit in \eqref{eq:def_surface_tension} exists and that it is the same if taken jointly in $L$ and $M$ for any $M=M(L)\to \infty$ as $L\to \infty$, see e.g. \cite{MMR92}. 

\begin{theorem}[\cite{LebPfi81}]\label{thm:pos_suf_tens}
    One has $\tau_\beta>0$ for every $\beta>\beta_c$.
\end{theorem}

\begin{proof}
    Our main tool will be the Ginibre inequality (see \cite{Leb77}):
    \begin{equation}\label{eq:Ginibre}
        \langle \sigma_A\sigma_B \rangle^{\eta'}_{\Lambda,\beta}-\langle \sigma_A\sigma_B \rangle^{\eta}_{\Lambda,\beta} \geq \big| \langle \sigma_A \rangle^{\eta'}_{\Lambda,\beta}\langle \sigma_B \rangle^{\eta}_{\Lambda,\beta} - \langle \sigma_B \rangle^{\eta'}_{\Lambda,\beta}\langle \sigma_A \rangle^{\eta}_{\Lambda,\beta} \big|,
    \end{equation}
    for all $A,B\subset \Lambda$ and $\eta,\eta'$ such that $|\eta|\leq \eta'$.
    Let $\tau^{L,M}_\beta \coloneqq \frac{1}{L^{d-1}} \log \frac{Z^\pm_{\cR(L,M),\beta}}{Z^+_{\cR(L,M),\beta}}$ and note that
    \begin{equation*}
    \frac{d}{d\beta} \tau^{L,M}_\beta =  \frac{1}{L^{d-1}} \sum_{x,y\in E(\cR(L,M))} \langle \sigma_x\sigma_y \rangle^{+}_{\cR(L,M),\beta} - \langle \sigma_x\sigma_y \rangle^{\pm}_{\cR(L,M),\beta}.
    \end{equation*}
    Let $\tau^{L}_\beta \coloneqq \lim_{M\to \infty} \tau^{L,M}_\beta$ and note that, since $\cR(L)=\{-L,L\}^{d-1}\times\Z$ is ``one-dimensional'', the terms in the sum tends to $0$ exponentially fast with the distance to $\Z^{d-1}\times\{0\}$. Also, by \eqref{eq:Ginibre} each term in the sum is positive. Therefore, by applying \eqref{eq:Ginibre}, disregarding the horizontal edges, taking $M\to\infty$ and using translation invariance in the $i_d$ direction, we obtain
    \begin{align*}
    \frac{d}{d\beta} \tau^{L}_\beta &\geq  \frac{1}{L^{d-1}} \sum_{\substack{ x\in \{-L,\ldots,L\}^{d-1}\\ j\in\Z}}  \langle \sigma_{(x,0)} \rangle^{+}_{\cR(L),\beta} \big( \langle \sigma_{(x,j+1)} \rangle^{\pm}_{\cR(L),\beta} - \langle \sigma_{(x,j)} \rangle^{\pm}_{\cR(L),\beta} \big) \\
    &= \frac{1}{L^{d-1}} \sum_{x\in \{-L,\ldots,L\}^{d-1}} 2 (\langle \sigma_{(x,0)} \rangle^{+}_{\cR(L),\beta})^2
    \end{align*}
    where we used that $\langle \sigma_{(x,j)} \rangle^{\pm}_{\cR(L),\beta} \to \langle \sigma_{(x,0)} \rangle^{+}_{\cR(L),\beta}$ and $j\to+\infty$ and $\langle \sigma_{(x,j)} \rangle^{\pm}_{\cR(L),\beta} \to - \langle \sigma_{(x,0)} \rangle^{+}_{\cR(L),\beta}$ as $j\to-\infty$. Finally, by taking $L\to\infty$ we obtain 
    \begin{equation*}
        \frac{d}{d\beta} \tau_\beta \geq 2^d (\langle \sigma_{0} \rangle^{+}_{\beta})^2,
    \end{equation*}
    which readily implies that $\tau_\beta>0$ for every $\beta>\beta_c$.
\end{proof}

The Ising model at inverse temperature $\beta$ and the FK-Ising with $p=1-e^{-2\beta}$ are coupled together through the so-called Edwards--Sokal coupling -- see \cite{Gri06,DumNotes17} for details. In particular, one has the relation $\phi^1_p[\lr{}{}{0}{\infty}]=\langle \sigma_0 \rangle^+_\beta$, and therefore $p_c=1-e^{-2\beta_c}$. One can also re-interpret the Ising surface tension in terms of the FK-Ising model using the aforementioned coupling and noting that
\begin{equation*}
    \frac{Z^\pm_{\cR(L,M),\beta}}{Z^+_{\cR(L,M),\beta}} = \phi^1_{\cR(L,M),p} [\nlr{}{}{\partial^+\cR(L,M)}{\partial^-\cR(L,M)}].
\end{equation*}
Therefore, in order to prove Theorem~\ref{thm:tilde_p_c=p_c}, it remains to compare wired and free boundary conditions, which is the subject of the following proposition.

\begin{proposition}[\cite{Bod05}]\label{prop:surface_tension_wired=free}
    For $q=2$, $\tau_\beta>0$ implies $\tilde{\tau}_p>0$ with $p=1-e^{-2\beta}$.
\end{proposition}

\begin{proof}
   The proof follows the same lines as \cite[Theorem 3.2]{Bod05}, with some simplifications.
   By definition and monotonicity in boundary conditions, if $\tau_\beta>0$ then for every $\delta>0$ one has
   \begin{equation*}
       \phi^1_{\cR(L,\delta L),p}[ \nlr{}{}{\partial^{-} \cR(L, \delta L)}{\partial^{+} \cR(L, \delta L)} ] \leq e^{-L^{d-1} \tau_\beta/2}
   \end{equation*}
  for $L$ large enough. Notice that the lateral boundary of $\cR(L,\delta L)$ has size of order $\delta L^{d-1}$, so for $\delta$ small enough a simple Radon–Nikodym derivative estimate gives
   \begin{equation*}
   \phi^{1,0}_{\cR,p}[ \nlr{}{}{\partial^{bot} \cR(L, \delta L)}{\partial^{top} \cR(L, \delta L)} ] \leq e^{-L^{d-1} \tau_\beta/4},
   \end{equation*}
   where $\cR=\cR(L,\delta L)$ and $\xi=\{1,0\}$ stands for the boundary condition which is wired on the top and bottom boundaries and free on the lateral boundary. By inclusion of events, we have 
   \begin{equation}\label{eq:mixed_b.c.}
   \phi^{1,0}_{\cR,p}[ \cD ] \leq e^{-L^{d-1} \tau_\beta/4},
   \end{equation}
   where $\cD\coloneqq \{ \nlr{}{}{\partial^{bot} \cR(L, \frac{\delta}{2} L)}{\partial^{top} \cR(L, \frac{\delta}{2} L)} \}$.
   For all $s\in[0,1]$, let $\phi^{s,0}_{\cR,p}$ be the measure with boundary condition $\{1,0\}$ and intensity $sp$ (defined in the natural way) on the set of bonds $\cB$ which are incident to $\partial^{top} \cR$ or $\partial^{bot} \cR$, and intensity $p$ elsewhere. Note that $\phi^{0,0}_{\cR,p}=\phi^{0}_{\cR(L,\delta L-1)}\prec \phi^0_{\Lambda_L,p}$.
   By the Russo's formula for FK percolation \cite[Theorem 3.12]{Gri06}, we have
   \begin{equation*}
       \partial_s (\log \phi^{s,0}_{\cR,p}[\cD])= \frac{1}{1-sp} \sum_{b\in \cB} \phi^{s,0}_{\cR,p}[\omega_b \,|\, \cD] -\phi^{s,0}_{\cR,p}[\omega_b].
   \end{equation*}
    Since $\cD$ is supported on $\cR(L,\frac{\delta}{2} L)$, the domain Markov property and a comparison between boundary conditions imply that, for every bond $b$ at distance at least $\frac{\delta}{2} L$ from the lateral boundary of $\cR$, we have
    $$|\phi^{s,0}_{\cR,p}[\omega_b \,|\, \cD] -\phi^{s,0}_{\cR,p}[\omega_b]| \leq \phi^{s,1}_{H(\frac{\delta}{2} L),p}[\omega_{b_0}]-\phi^{s,0}_{H(\frac{\delta}{2} L),p}[\omega_{b_0}],$$ 
    where $b_0$ is the bond $\{0,e_d\}$ and $\phi^{s,1}_{H(K),p}$ (resp. $\phi^{s,0}_{H(K),p}$) is the measure on $H(K)\coloneqq \{-K,\ldots, K\}^{d-1}\times \{0,\ldots, K\}$ with intensity $sp$ and wired boundary conditions on the bottom face and wired (resp. free) boundary condition on the rest of the boundary. Our goal is to show that $\phi^{s,1}_{H(K),p}[\omega_{b_0}]-\phi^{s,0}_{H(K),p}[\omega_{b_0}]$ is small for large $K$. This is given by the following ``weak mixing'' result.

    \begin{lemma}\label{lem:wired=free}
       For every $s>0$ and $p\in[0,1]$, one has $\phi^{s,1}_{H(K),p}[\omega_{b_0}]-\phi^{s,0}_{H(K),p}[\omega_{b_0}]\to 0$ as $K\to\infty$.
    \end{lemma}

    Before proving Lemma~\ref{lem:wired=free}, we conclude the proof of Proposition~\ref{prop:surface_tension_wired=free}. By Lemma~\ref{lem:wired=free} and the dominated convergence theorem, we have
    \begin{align}\label{eq:dif_log}
    \begin{split}
        &|\log \phi^{1,0}_{\cR,p}[\cD] - \log \phi^{0,0}_{\cR,p}[\cD]|=\int_{0}^1 \partial_s (\log \phi^{s,0}_{\cR,p}[\cD]) ds \\
        &\leq C\delta L^{d-1} + CL^{d-1} \int_{0}^1 (\phi^{s,1}_{H(\frac{\delta}{2} L),p}[\omega_{b_0}]-\phi^{s,0}_{H(\frac{\delta}{2} L),p}[\omega_{b_0}]) ds \leq  2C\delta L^{d-1},
    \end{split}
    \end{align}
    for $L$ large enough, where the first term $C\delta L^{d-1}$ accounts for the bonds with distance at most $\frac{\delta}{2} L$ from the lateral boundary of $\cR$. By choosing $\delta>0$ small enough so that $2C\delta<\tau_\beta/4$, inequalities \eqref{eq:mixed_b.c.} and \eqref{eq:dif_log} imply the result.
\end{proof}

\begin{proof}[Proof of Lemma~\ref{lem:wired=free}]
    Setting $\beta=-\tfrac{1}{2}\log(1-p)$ and $h=-\tfrac{1}{2}\log(1-sp)/\beta$, it follows easily from the Edwards--Sokal coupling that the statement is equivalent to
    \begin{equation*}
        \langle \sigma_0 \rangle^{h,+}_{H(K),\beta} - \langle \sigma_0 \rangle^{h,0}_{H(K),\beta} \to 0 ~~~ \text{as } K\to\infty,
    \end{equation*}
    where $h,+$ (resp. $h,0$) denotes the boundary field equals  $h$ on the bottom face of $H(K)$ and $1$ (resp. $0$) on the rest of $\partial H(K)$. It follows from \cite{MMP84} that the limits $\langle \sigma_0 \rangle^{h,+}_{\H,\beta} = \lim_{K\to\infty} \langle \sigma_0 \rangle^{h,+}_{H(K),\beta}$ and $\langle \sigma_0 \rangle^{h,0}_{\H,\beta} = \lim_{K\to\infty} \langle \sigma_0 \rangle^{h,0}_{H(K),\beta}$ exist and are analytic in $h>0$. By \cite{FP87} (see (2.13) and (2.21)) one has $\langle \sigma_0 \rangle^{h,+}_{\H,\beta} = \langle \sigma_0 \rangle^{h,-}_{\H,\beta}$ for all $h\geq1$. By monotonicity in boundary conditions it follows that $\langle \sigma_0 \rangle^{h,+}_{\H,\beta} = \langle \sigma_0 \rangle^{h,0}_{\H,\beta}$ for all $h\geq1$, thus for all $h>0$ by analyticity, which completes the proof.
\end{proof}

\bibliographystyle{alpha}

\end{document}